\documentclass[11pt]{amsart}

\usepackage{amsmath,amssymb,latexsym,soul,cite,mathrsfs}

\usepackage{color,enumitem,graphicx}
\usepackage[colorlinks=true,urlcolor=blue,
citecolor=red,linkcolor=blue,linktocpage,pdfpagelabels,
bookmarksnumbered,bookmarksopen]{hyperref}
\usepackage[english]{babel}

\usepackage[left=2.9cm,right=2.9cm,top=2.8cm,bottom=2.8cm]{geometry}
\usepackage[hyperpageref]{backref}

\usepackage[colorinlistoftodos]{todonotes}
\makeatletter
\providecommand\@dotsep{5}
\def\listtodoname{List of Todos}
\def\listoftodos{\@starttoc{tdo}\listtodoname}
\makeatother

\numberwithin{equation}{section}

\newtheorem{theorem}{Theorem}[section]

\newtheorem{lemma}[theorem]{Lemma}
\newtheorem{corollary}[theorem]{Corollary}

\newcommand\R{\mathbb R}

\begin{document}

\title[Multiple solutions for two classes of quasilinear problems defined on .... ]
{Multiple solutions for two classes of quasilinear problems defined on a nonreflexive Orlicz-Sobolev space}
\author{Claudianor O. Alves, Sabri Bahrouni and Marcos L. M. Carvalho}
\address[Claudianor O. Alves ]
{\newline\indent Unidade Acad\^emica de Matem\'atica
	\newline\indent
	Universidade Federal de Campina Grande
\newline\indent
e-mail:coalves@mat.ufcg.edu.br
\newline\indent	
58429-970, Campina Grande - PB, Brazil}

\address[Sabri Bahrouni]
{\newline\indent Mathematics Department, Faculty of Sciences
	\newline\indent
	University of Monastir
	\newline\indent
	e-mail:sabribahrouni@gmail.com
	\newline\indent	
	5019 Monastir, Tunisia}

\address[ Marcos Carvalho]
{\newline\indent Instituto de Matem\'atica e Estat\'istica
	\newline\indent
	Universidade Federal de Goias
\newline\indent
e-mail: marcos$\_$leandro$\_$carvalho@ufg.br
\newline\indent 74001-970, Goi\^ania, GO, Brazil }

\pretolerance10000

\begin{abstract}
	In this paper we prove  the existence and multiplicity of solutions for a large class of quasilinear problems on a nonreflexive  Orlicz-Sobolev space. Here, we use the variational methods developed by Szulkin \cite{Szulkin} combined with some properties of the weak$^*$ topology.
\end{abstract}

\thanks{ C. O. Alves was partially
	supported by  CNPq/Brazil 304804/2017-7 }
\subjclass[2010]{35A15, 35J62, 46E30 }
\keywords{Orlicz-Sobolev Spaces, Variational Methods, Quasilinear Elliptic Problems, $\Delta_{2}$-condition, Modular}

\maketitle	

\section{Introduction}
This paper concerns the existence and multiplicity of weak solutions for a class of quasilinear elliptic problem of the type
$$
\left\{
\begin{array}{ll}
\displaystyle-\Delta_{\Phi}{u}=\lambda f(x,u), \quad \mbox{in} \quad \Omega, \\
u =0,\quad\mbox{on} \quad \partial \Omega,
\end{array}
\right.
\leqno{(P)}
$$
where $\Omega \subset \mathbb{R}^N$  is a smooth bounded domain, $N \geq 1$, $\lambda$ is a positive parameter and $f:\Omega\times\mathbb{R}\to\mathbb{R}$ is a Carath\'eodory function verifying some conditions that will be mentioned later on. It is important to recall that
$$
\Delta_\Phi u = \mbox{div}(\phi(|\nabla u|)\nabla u),
$$
where $\Phi:\mathbb{R} \to \mathbb{R}$ is a N-function of the form
$$
\Phi(t)=\int_{0}^{|t|}s\phi(s) \,ds
$$
and $\phi:(0, +\infty) \to (0, +\infty) $ is a $C^{1}$ function verifying some technical assumptions.

We would like to detach that this type of operator appears in a lot of physical applications, such as: Nonlinear Elasticity,  Plasticity, Generalized Newtonian Fluid, Non-Newtonian Fluid and Plasma Physics. For more details about the physical applications we cite \cite{Db}, \cite{FN2} and their references.

Motivated by above applications, many authors have studied problems involving quasilinear problem driven by a N-function $\Phi$, we would like to cite  Bonanno, Bisci and Radulescu \cite{BBR, BBR2}, Cerny \cite{Cerny}, Cl\'ement, Garcia-Huidobro and Man\'asevich \cite{VGMS}, Donaldson \cite{donaldson}, Fuchs and Li \cite{Fuchs1},  Fuchs and Osmolovski \cite{Fuchs2}, Fukagai, Ito and Narukawa \cite{FN}, Gossez \cite{gossez2}, Le and Schmitt \cite{LK}, Mihailescu and Radulescu \cite{MR1, MR2}, Mihailescu and Repovs \cite{MD}, Mihailescu, Radulescu and Repovs \cite{MRR}, Mustonen and Tienari \cite{MT}, Orlicz \cite{O} and their references, where quasilinear problems like $(P)$ have been considered in bounded and unbounded domains of $\mathbb{R}^{N}$.

 In all of these works the so called $\Delta_{2}$-condition has been assumed on $\Phi$ and $\widetilde{\Phi}$, which ensures that the Orlicz-Sobolev space $W^{1,\Phi}(\Omega)$ is a reflexive Banach space. This assertion is used several times in order to get a nontrivial solution for elliptic problems taking into account the weak topology and the classical variational methods to $C^1$ functionals.

In recent years many researchers have studied the nonreflexive case, which is more subtle from a mathematical point of view, because in general the energy functional associated with these problems are in general only continuous and the classical variational methods to $C^1$ cannot be used. For example, in \cite{GKMS}, Garc\'ia-Huidobro, Khoi, Man\'asevich and Schmitt considered the existence of solution for the following nonlinear eigenvalue problem
\begin{equation} \label{P1}
	\left\{
	\begin{array}{ll}
		-\Delta_{\Phi}{u}=\lambda \Psi (u), \quad \mbox{in} \quad \Omega \\
		u=0, \quad \mbox{on} \quad \partial \Omega,
	\end{array}
	\right.
\end{equation}
where $\Omega$ is a bounded domain, $\Phi:\mathbb{R} \to \mathbb{R}$ is a N-function and $\Psi:\mathbb{R} \to \mathbb{R}$ is a continuous function verifying some others technical conditions. In that paper, the authors studied the situation where $\Phi$ does not satisfy the well known $\Delta_2$-condition. More precisely, in the first part of that paper the authors consider the function
\begin{equation} \label{phi*}
	\Phi(t)= (e^{t^{2}}-1)/2, \quad \forall t \in \mathbb{R}.
\end{equation}
More recently, Bocea and  Mih\u{a}ilescu \cite{BM}  made a careful study about the eigenvalues of the problem
\begin{equation} \label{P2}
	\left\{
	\begin{array}{ll}
		-div(e^{|\nabla u|^{2}}\nabla u)-\Delta u=\lambda u, \quad \mbox{in} \quad \Omega \\
		u=0, \quad \mbox{on} \quad \partial \Omega.
	\end{array}
	\right.
\end{equation}
After that, Silva, Gon\c calves and Silva \cite{EGS} considered existence of multiple solutions for a class of problem like (\ref{P1}). In that paper the $\Delta_2$-condition is not also assumed and the main tool used was the truncation of the nonlinearity together with a minimization procedure for the energy functional associated to the quasilinear elliptic problem (\ref{P1}).

In \cite{DMKV}, Silva, Carvalho, Silva and Gon\c calves study a class of problem (\ref{P1}) where the energy functional satisfies the mountain pass geometry and the N-function $\widetilde{\Phi}$ does not satisfies the $\Delta_2$-condition and has a polynomial growth. Still related to the mountain pass geometry, in \cite{AEM}, Alves, Silva and Pimenta also considered the problem (\ref{P1}) for a large class of function $\Psi$, but supposing that $\Phi$ has an exponential growth like (\ref{phi*}).

Motivated by above study involving nonreflexive Banach spaces, we intend to consider two new classes of problem $(P)$ where $W_0^{1,\Phi}(\Omega)$ can be nonreflexive. The plan of the paper is as follows: In Section 2 we done a review about the main properties involving the Orlicz-Sobolev spaces that will be used in our approach. In Section 3 we consider our first class of problem  assuming the following conditions:
$$
t \mapsto t\phi(t); \quad t>0 \;\; \quad \mbox{increasing;}\leqno{(\phi_1)}
$$
$$
\lim_{t \to 0}t\phi(t)=0, \quad  \lim_{t \to +\infty}t\phi(t)=+\infty. \leqno{(\phi_2)}
$$
$$
t\mapsto t^2\phi(t)\qquad\mbox{is convex and}\qquad\frac{t^{2}{\phi(t)}}{\Phi(t)} \geq l >1, \quad \forall t >0. \leqno{(\phi_3)}
$$

	 The Carath\'eodory function $f:\Omega\times\mathbb{R}\to\mathbb{R}$ satisfies:
		\begin{itemize}
			\item[$(f_0)$] There exist a constant $C>0$ and a function $a:[0, +\infty) \to (0, +\infty) $ such that
							$$|f(x,t)|\leq C(a(t)t+1),~\mbox{a.e. }\ x\in\Omega,~t\in[0,\infty),$$
							where
							$$A(t)=\int_0^ta(s)sds$$
							is a N-function satisfying $1<\displaystyle m_A:=\sup_{t>0}\frac{a(t)t^2}{A(t)}<l.$
			\item[$(f_1)$] There exits $\delta>0$ such that $t\mapsto F(x,t):=\int_0^{t}f(x,s)ds$ is decreasing in $[0,\delta)$ a.e. in $\Omega$;
			\item[$(f_2)$] There exits $t_1>0$ such that
			$$F(x,t_1)>0,~\mbox{a.e. in }\Omega.$$
		\end{itemize}

		 A model of nonlinearity satisfying $(f_0)-(f_2)$ is
			$$
			f(x,t)=t_+^{p-1}-t_+^{q-1}, \quad t \in \mathbb{R},
			$$
			where $t_+=\max\{t,0\}$, $1<q<p$, which was considered in \cite{MRep}.
			However, our hypothesis are more general than those in \cite{MRep}, because in our case, $\Phi$ does not satisfies the $\Delta_2$-condition and $t\mapsto\Phi(\sqrt{t})$ does not need be a convex function. Another model of nonlinearity satisfying $(f_0)-(f_2)$, which was not treated in \cite{MRep}, is
			$$
			f_1(x,t)=pt_+^{p-1}\ln(1+t_+)-\frac{t^p}{\ln(1+t_+)}-qt_+^{q-1}, \quad t \in \mathbb{R},
			$$
			where $1<q<p$. Related to the conditions $(\phi_1)-(\phi_3)$, it is possible to show that the functions below  satisfy these conditions: \\
			\noindent $(i)$ \, $ \Phi(t) = (1+|t|^{2})^{\alpha}-1, \alpha \in (1, \frac{N}{N-2})$, \\
			\noindent $(ii)$ $ \Phi(t) = t^{p}\ln(1+|t|), 1< \frac{-1+\sqrt{1+4N}}{2}<p<N-1, N\geq 3$, \\
			\noindent  $(iii)$ $\Phi(t) = \int_{0}^{|t|}s^{1-\alpha}(\sinh^{-1}s)^{\beta}ds, 0\leq \alpha \leq 1$ and $\beta > 0$, \\
			\noindent $(iv)$ $\Phi(t) = \frac{1}{p}|t|^{p}$ for $p>1$,\\
			\noindent $(v)$ $\Phi(t) = \frac{1}{p}|t|^{p} + \frac{1}{q}|t|^{q}$ where $1<p<q<N$ with $q \in (p, p^{*}),$ \\
			\noindent and \\
			\noindent  $(vi)$ $\Phi(t)=(e^{|t|^{2}}-1)/2$. \\

		From now on, we say that $u \in W_0^{1,\Phi}(\Omega)$ is a weak solution of $(P)$ whenever
$$
\int_{\Omega}\phi(|\nabla u|)\nabla u \nabla v\,dx=\lambda\int_{\Omega}f(x,u)v\,dx, \quad \forall v \in W_0^{1,\Phi}(\Omega).
$$

Under these assumptions the main result in this section can be stated as follows:

\begin{theorem} \label{T1} Assume $(f_0)-(f_2)$ and $(\phi_1) - (\phi_3)$. Then there exists $\lambda^*>0$ such that problem $(P)$ has at least two nontrivial weak solutions for all $\lambda>\lambda^*$.
\end{theorem}

In Section 4, we study a second class of problem, where we require the following structural assumptions on $\Phi$ and $f$:
\begin{itemize}
    \item[$(\phi_4)$ ] $\displaystyle 0\leq\ell -1 =\inf_{t>0}\frac {(t\phi(t))^\prime}{\phi(t)}\leq \frac {(t\phi(t))^\prime}{\phi(t)}\leq m-1,~t>0$
	\item[$(f_3)$] There exist $C>0$ and $0<\alpha<1$ such that
$$
f(x,0)\in L^\infty(\Omega)\quad\text{and}\quad |F(x,t)|\leq C {\Phi(t)^\alpha}, ~t\in\mathbb{R} \setminus \{0\}.
$$
\end{itemize}
The condition $(\phi_4)$ does not guarantee that $W_0^{1,\Phi}(\Omega)$ is reflexive, because it permits to work with the case $\ell=1$, where we have a loss of reflexivity, because in this situation $\widetilde{\Phi}$ does not satisfies the $\Delta_2$-condition. For $\ell=1$, we have as model $\Phi(t)=|t|\log(1+|t|)$. In this section our main result is the following
\begin{theorem}\label{T2}
	Assume $(\phi_1),(\phi_2),(\phi_4)$, $(f_{1})- (f_{3})$. Then there exist $\lambda^*>0$ such that problem $(P)$ has at least two solutions $u_1,u_2\in  W_0^{1,\Phi}(\Omega)\cap L^{\infty}(\Omega) \setminus\{0\}$
	for all $\lambda>\lambda^*$.
\end{theorem}

Before concluding this introduction, we would like to point out that Theorems \ref{T1} and \ref{T2} complement the study made in \cite{MRep}, in the sense that we are considering new classes of N-functions that were not considered in that reference. Moreover, the above theorems are the first results in the literature involving multiplicity of solutions for a class of quasilinear problems driven by a N-function $\Phi$ whose the Orlicz-Sobolev space $W_0^{1,\Phi}(\Omega)$ can be nonreflexive.

\section{Basics on Orlicz-Sobolev spaces}

In this section we recall some properties of Orlicz and Orlicz-Sobolev spaces, which can be found in \cite{Adams, RR}. First of all, we recall that a continuous function $\Phi : \mathbb{R} \rightarrow [0,+\infty)$ is a
N-function if:
\begin{itemize}
	\item[$(i)$] $\Phi$ is convex.
	\item[$(ii)$] $\Phi(t) = 0 \Leftrightarrow t = 0 $.
	\item[$(iii)$] $\displaystyle\lim_{t\rightarrow0}\frac{\Phi(t)}{t}=0$ and $\displaystyle\lim_{t\rightarrow+\infty}\frac{\Phi(t)}{t}= +\infty$ .
	\item[$(iv)$] $\Phi$ is even.
\end{itemize}
We say that a N-function $\Phi$ verifies the $\Delta_{2}$-condition, if
\[
\Phi(2t) \leq K\Phi(t),\quad \forall t\geq 0,
\]
for some constant $K > 0$. For instance, it can be shown that functions  $\Phi$ given in $(i)-(v)$ satisfy the $\Delta_2$-condition, while $\Phi(t)=(e^{t^{2}}-1)/2$ does not verify it.

In what follows, fixed an open set $\Omega \subset \mathbb{R}^{N}$ and a N-function $\Phi$, we define the Orlicz space associated with $\Phi$ as
\[
L^{\Phi}(\Omega) = \left\{  u \in L_{loc}^{1}(\Omega) \colon \ \int_{\Omega} \Phi\left(\frac{|u|}{\lambda}\right)dx < + \infty \ \ \mbox{for some}\ \ \lambda >0 \right\}.
\]
The space $L^{\Phi}(\Omega)$ is a Banach space endowed with the Luxemburg norm given by
\[
\Vert u \Vert_{\Phi} = \inf\left\{  \lambda > 0 : \int_{\Omega}\Phi\Big(\frac{|u|}{\lambda}\Big)dx \leq1\right\}.
\]
The complementary function $\widetilde{\Phi}$ associated with $\Phi$ is given
by its Legendre's transformation, that is,
\[
\widetilde{\Phi}(s) = \max_{t\geq 0}\{ st - \Phi(t)\}, \quad  \mbox{for} \quad s\geq0.
\]
The functions $\Phi$ and $\widetilde{\Phi}$ are complementary each other. Moreover, we also have a Young type inequality given by
\[
st \leq \Phi(t) + \widetilde{\Phi}(s), \quad \forall s, t\geq0.
\]
Using the above inequality, it is possible to prove a H\"older type inequality, that is,
\[
\Big| \int_{\Omega}uvdx \Big| \leq 2 \Vert u \Vert_{\Phi}\Vert v \Vert_{ \widetilde{\Phi}},\quad \forall u \in L^{\Phi}(\Omega) \quad \mbox{and} \quad \forall v \in L^{\widetilde{\Phi}}(\Omega).
\]
The corresponding Orlicz-Sobolev space is defined by
\[
W^{1, \Phi}(\Omega) = \Big\{ u \in L^{\Phi}(\Omega) \ :\ \frac{\partial u}{\partial x_{i}} \in L^{\Phi}(\Omega), \quad i = 1, ..., N\Big\},
\]
endowed with the norm
\[
\Vert u \Vert_{1,\Phi} = \Vert \nabla u \Vert_{\Phi} + \Vert u \Vert_{\Phi}.
\]

The space $W_0^{1,\Phi}(\Omega)$ is defined as the weak$^*$ closure of $C_0^{\infty}(\Omega)$ in $W^{1,\Phi}(\Omega)$. Moreover, by the Modular Poincar\'e's inequality
$$
\int_{\Omega}\Phi(|u|)\,dx\leq \int_{\Omega}\Phi(d|\nabla u|)\,dx, \quad \forall u \in W_0^{1,\Phi}(\Omega),
$$
where $d=diam(\Omega)$, and it follows that
$$
\|u\|_\Phi\leq 2d\|\nabla u\|_\Phi,\qquad u\in W_0^{1,\Phi}(\Omega).
$$
The last inequality yields that  the functional $\|\cdot\|:=\|\nabla \cdot\|_\Phi$ defines an equivalent norm in $W_0^{1,\Phi}(\Omega)$. Here we refer the readers to the important works \cite{gossez,gossez2}. The spaces $L^{\Phi}(\Omega)$, $W^{1, \Phi}(\Omega)$ and $W_0^{1, \Phi}(\Omega)$ are separable and reflexive, when $\Phi$ and $\widetilde{\Phi}$ satisfy  $\Delta_{2}$-condition.

If $E^{\Phi}(\Omega)$ denotes the closure of $L^{\infty}(\Omega)$ in $L^{\Phi}(\Omega)$ with respect to the norm $\|\,\,\|_{\Phi}$, then  $L^{\Phi}(\Omega)$ is the dual space of $E^{\widetilde{\Phi}}(\Omega)$, while $L^{\widetilde{\Phi}}(\Omega)$ is the dual space of $E^{\Phi}(\Omega)$. Moreover, $E^{\Phi}(\Omega)$ and $E^{\widetilde{\Phi}}(\Omega)$ are separable spaces and any continuous linear functional $M:E^{\Phi}(\Omega) \to \mathbb{R}$ is of the form
$$
M(v)=\int_{\Omega}v(x)g(x)\,dx \quad \mbox{for some} \quad g \in L^{\widetilde{\Phi}}(\Omega).
$$
We recall that if $\Phi$ verifies the $\Delta_2$-condition, we then have $E^{\Phi}(\Omega)=L^{\Phi}(\Omega)$.

The next result is crucial in the approach explored in Section 3, and its proof follows directly from a result by Donaldson \cite[Proposition 1.1]{donaldson}.

\begin{lemma}  \label{Estrela} Assume that $\Phi$ is a N-function and $\widetilde{\Phi}$ verifies the $\Delta_2$-condition. If $(u_n) \subset W^{1,\Phi}_0(\Omega) $ is a bounded sequence, then there are a subsequence of $(u_n)$, still denoted by itself, and $u \in  W^{1,\Phi}_0(\Omega)$ such that
	$$
	u_n \stackrel{*}{\rightharpoonup} u \quad \mbox{in} \quad W^{1,\Phi}_0(\Omega)
	$$
	and
	$$
	\int_{\Omega}u_n v \,dx  \to  \int_{\Omega}u v \,dx, \quad  \int_{\Omega}\frac{\partial u_n}{\partial x_i} w \,dx  \to  \int_{\Omega}\frac{\partial u}{\partial x_i} w \,dx, \quad \forall v,w \in  E^{\tilde{\Phi}}(\Omega)=L^{\tilde{\Phi}}(\Omega).
	$$
\end{lemma}	

The lemma just above is crucial when the space $W^{1,\Phi}_0(\Omega)$ is not reflexive, for example if $\Phi(t)=(e^{t^2}-1)/2$. However, if $\Phi$ is one of the functions given in $(i)-(v)$, the above lemma is not necessary since $\Phi$ and $\widetilde{\Phi}$ satisfy the $\Delta_2$-condition, and so, $W^{1,\Phi}_0(\Omega)$ is reflexive. Here we would like to point out that $(\phi_3)$ ensures that $\widetilde{\Phi}$ verifies the $\Delta_2$-condition, for more details see Fukagai and Narukawa \cite{FN}.

\begin{lemma}\label{ed1}
	Suppose $(\phi_{1}),~(\phi_{2})$ and either $(\phi_{3})$ or $(\phi_{4})$. Let $(u_{n}) \subset W_0^{1,\Phi}(\Omega)$ be a fixed sequence such that $\|u_{n}\| \rightarrow \infty$. Then there exists $n_0 \in \mathbb{N}$ such that
	$$
	\int_{\Omega} \Phi(|\nabla u_n|) dx \geq \|u_n\|^l, \quad \forall n \geq n_0.
	$$
	Hence
	$$
	\|u_n\| \to +\infty \,\, \mbox{implies that} \,\, \int_{\Omega}  \Phi(|\nabla u_n|)\, dx \to +\infty.
	$$
\end{lemma}
\begin{proof}
	The proof is similar to that given in \cite[Lemma 2.1]{FN}.
\end{proof}

\section{Proof of Theorem \ref{T1}}
Note that under hypotheses $(\phi_1)-(\phi_3)$, it is well known that $\Phi$ does not satisfy the $\Delta_2$-condition, and as a consequence, $ W^{1,\Phi}_0(\Omega)$ is nonreflexive anymore. Under these conditions,  it is also well known that there exists $u \in W_0^{1,\Phi}(\Omega)$ such that
$$
\int_{\Omega}\Phi(|\nabla u|)\,dx=+ \infty.
$$
However, independent of the $\Delta_2$-condition, $(f_0)$ guarantees that the embedding \linebreak $W_0^{1,\Phi}(\Omega)\hookrightarrow L^A(\Omega)$ is continuous. Having this in mind, the energy functional \linebreak $I: W_0^{1,\Phi}(\Omega) \to \mathbb{R}\cup \{+\infty\}$ associated with $(P)$ given by
\begin{equation} \label{I}
I(u)=\int_{\Omega}\Phi(|\nabla u|)dx -\lambda\int_{\Omega}F(x,u) dx,~u\in  W_0^{1,\Phi}(\Omega)
\end{equation}
is well defined.  Hereafter, we denote by $D_\Phi \subset  W_0^{1,\Phi}(\Omega)$ the  set
$$
D_\Phi=\left\{u \in W_0^{1,\Phi}(\Omega)\,:\, \int_{\Omega}\Phi(|\nabla u|)\,dx<+\infty\right\}.
$$
The reader is invited to observe that $D_\Phi= W_0^{1,\Phi}(\Omega)$ when $\Phi$ satisfies the $\Delta_2$-condition.

As an immediate consequence of the above remarks, we cannot guarantee that $I$ belongs to $C^{1}(W_0^{1,\Phi}(\Omega),\mathbb{R})$. However, the functional $\mathcal{F}: W_0^{1,\Phi}(\Omega) \to \mathbb{R} $ given by
$$
\mathcal{F}(u)=\int_{\Omega}F(x,u) dx
$$
belongs to $C^{1}( W_0^{1,\Phi}(\Omega),\mathbb{R})$ and its derivative is given by
$$
\mathcal{F}'(u)v=\int_{\mathbb{R}^N}f(x,u)v \,dx, \quad \forall u,v \in  W_0^{1,\Phi}(\Omega).
$$

Related to the functional $Q: W_0^{1,\Phi}(\Omega) \to \mathbb{R} \cup\{+\infty\}$ given by
\begin{equation} \label{Q}
Q(u)=\int_{\Omega}\Phi(|\nabla u|)dx,
\end{equation}
we know that it is strictly convex   and  l.s.c. with respect to the weak$^*$ topology. Furthermore, $Q \in C^{1}(W_0^{1,\Phi}(\Omega),\mathbb{R})$ when $\Phi$ and $\widetilde\Phi$ satisfy the $\Delta_2$-condition.

From the above commentaries, in the present paper we will use a minimax method developed by Szulkin  \cite{Szulkin}. In this sense, we will say that $u \in D_\Phi$ is a critical point for $I$ if $0 \in \partial I(u) = \partial Q(u) - \mathcal{F}'(u)$, since $\mathcal{F} \in C^{1}( W_0^{1,\Phi}(\Omega),\mathbb{R})$. Then $u \in D_\Phi$ is a critical point of $I$ if, and only if, $\mathcal{F}'(u) \in \partial Q(u)$, what, since $Q$ is convex, is equivalent to
\begin{equation} \label{E1}
Q(v)-Q(u)\geq \lambda\int_{\Omega}f(x,u)(v-u)\,dx, \quad \forall v \in   W_0^{1,\Phi}(\Omega).
\end{equation}

In the case where  $\Phi$ and $\widetilde\Phi$ satisfy the $\Delta_2$-condition, we would like to point out that the energy functional $I \in C^{1}(  W_0^{1,\Phi}(\Omega),\mathbb{R})$, and the last inequality is equivalent to
\begin{equation} \label{E1.0}
I'(u)v=0, \quad \forall v \in   W_0^{1,\Phi}(\Omega),
\end{equation}
that is,
$$
\int_{\Omega}\phi(|\nabla u|)\nabla u \nabla v\,dx=\lambda\int_{\Omega}f(x,u)v\,dx, \quad \forall v \in  W_0^{1,\Phi}(\Omega).
$$
The last identity ensures that $u$ is a weak solution of $(P)$. However, when $\Phi$ does not satisfy the $\Delta_2$-condition, the above conclusion is not immediate, and a careful analysis must be done. For more details see Lemma \ref{Lema0} below.

From now on, let us denote by $\| \cdot\|$ the usual norm in  $ W_0^{1,\Phi}(\Omega)$ given by
$$
\|u\|=\|\nabla u\|_{\Phi}
$$
where
$$
\|\nabla u\|_{\Phi}=\inf\left\{\lambda >0\,:\, \int_{\Omega}\Phi\left(\frac{| \nabla u|}{\lambda}\right)\,dx \leq 1\right\}.
$$
Moreover, we also denote by $\text{dom}(\phi(t)t) \subset  W_0^{1,\Phi}(\Omega)$ the following set
$$
\text{dom}(\phi(t)t)=\left\{ u \in  W_0^{1,\Phi}(\Omega)\,:\, \int_{\Omega}\widetilde{\Phi}(\phi(|\nabla u|)|\nabla u|)\,dx<\infty \right\}.
$$
As $\widetilde{\Phi}$ verifies the $\Delta_2$-condition, the above set can be written of the form
$$
\text{dom}(\phi(t)t)=\left\{ u \in   W_0^{1,\Phi}(\Omega)\,:\, \phi(|\nabla u|)|\nabla u|  \in L^{\widetilde{\Phi}}(\Omega) \right\}.
$$
The set $\text{dom}(\phi(t)t)$ is not empty, since it is easy to see that $C_0^{\infty}(\Omega) \subset \text{dom}(\phi(t)t)$.

\begin{lemma}\label{dominio} Suppose $(\phi_{1})-(\phi_{2})$. For each $u \in D_\Phi$, there is a sequence $(u_n) \subset  \text{dom}(\phi(t)t)$ such that
	$$
|u_n|\leq |u|,\qquad\int_{\Omega}\Phi(|\nabla u_n|)\,dx \leq \int_{\Omega}\Phi(|\nabla u|)\,dx  \quad \mbox{and} \quad \|u-u_n\| \leq {1}/{n}.
	$$	
\end{lemma}
\begin{proof}
	See \cite[Lemma 3.2]{AEM}.
\end{proof}

\begin{lemma}\label{I-coercive}
	Assume that $(\phi_1)-(\phi_3)$ and $(f_0)$ hold. Then, functional $I$ is coercive.
\end{lemma}
\begin{proof}
	Indeed, supposing  $\|u\|\geq 1$, using $(f_0)$ and the embeddings $W_0^{1,\Phi}(\Omega)\hookrightarrow L_A(\Omega)$ and $W_0^{1,\Phi}(\Omega) \hookrightarrow L^1(\Omega)$, we conclude that there exist positive constants $C_1$ and $C_2$ satisfying
	$$
	I(u)\geq \|u\|^l- \lambda C_1\|u\|^{m_A}-\lambda C_2\|u\|.
	$$
Recalling the $l>m_A>1$, we get the desired result.
\end{proof}

\begin{lemma}\label{negative}
	Assume that $(\phi_1)-(\phi_3)$ and $(f_0)$, $(f_2)$ hold. Then, there exists $\lambda_*>0$ such that $I$ is bounded from below in  $W^{1,\Phi}_0(\Omega)$ and  $\displaystyle \inf_{u \in W_0^{1,\Phi}(\Omega)} I(u)<0$ for all $\lambda>\lambda_*$.
\end{lemma}	
\begin{proof} First of all, the fact that $I$ is coercive yields that there is $R>0$ such that
$$
I(u) \geq 1, \quad \mbox{for} \quad \|u\| \geq R.
$$	
On the other hand, by definition of $I$, a simple computation gives
$$
|I(u)| \leq K, \quad \mbox{for} \quad \|u\| \leq R.
$$	
From this, there is $M>0$ such that
$$
I(u) \geq -M, \quad \mbox{for all} \quad u \in W_0^{1,\Phi}(\Omega),
$$	
showing the boundedness of $I$ from below in $W_0^{1,\Phi}(\Omega)$. 	
	
Next, we will show that $\displaystyle \inf_{u \in W_0^{1,\Phi}(\Omega)} I(u)<0$. By $(f_2)$, there exists $t_1>0$ such that $F(x,t_1)>0$. Let $\Omega_1\subset \Omega$ be a compact subset large enough and $u_0\in W_0^{1,\Phi}(\Omega)$ such that $u_0(x)=t_1$ in $\Omega_1$ and $0\leq u_0(x)\leq t_1$ in $\Omega\setminus\Omega_1$. Note that $\{x\in\Omega:F(x,u_0(x))<0\}\subset \Omega\setminus\Omega_1$. So, using $(f_0)$,
\begin{eqnarray}\label{Negative}
	\int_{\Omega}F(x,u_0)dx
	&\geq & \int_{\Omega_1}F(x,t_1)dx-C\int_{\Omega\setminus\Omega_1}(A(u_0)+|u_0|)dx\nonumber\\
	&\geq & \int_{\Omega_1}F(x,t_1)dx-C|\Omega\setminus\Omega_1|(A(t_1)+|t_1|)>0	
\end{eqnarray}
provided that $|\Omega\setminus\Omega_1|$ is small enough. Thus $I(u_0) < 0$ for $\lambda > 0$ large
enough. This proves the lemma.
\end{proof}
From Lemmas \ref{I-coercive} and \ref{negative},  $I$ is bounded from below in  $W^{1,\Phi}_0(\Omega)$. Thereby, there is $(u_n) \subset W^{1,\Phi}_0(\Omega)$ such that
$$
I(u_n) \to I_{\infty}=\inf_{u \in W^{1,\Phi}_0(\Omega)}I(u) \quad \mbox{as} \quad n \to +\infty.
$$
Consequently, taking into account that $I$ is coercive, the sequence $(u_n)$ must be bounded in $W^{1,\Phi}_0(\Omega)$. Therefore, by Lemma \ref{Estrela}, for some subsequence denoted by itself, we obtain
$$
u_n \stackrel{*}{\rightharpoonup} u_1 \quad \mbox{in} \quad  W^{1,\Phi}_0(\Omega).
$$
Now, applying \cite[Lemma 3.2]{GKMS} and \cite{ET}, it follows that $I$ is weak$^*$ lower semicontinuous. As a consequence,
$$
\liminf_{n \to +\infty}I(u_n) \geq I(u_1).
$$
The last estimate implies that
$$
I(u_1)=I_{\infty}.
$$
From this, $u_1\in W_0^{1,\Phi}(\Omega)\setminus\{0\}$, $I(u_1)<0$ and
\begin{equation} \label{sol-u_1}
Q(v) -Q(u_1)\geq \lambda\int_{\Omega}f(x,u_1)(v-u_1)\,dx, \quad \forall v \in   W_0^{1,\Phi}(\Omega).
\end{equation}
\begin{lemma}\label{Domin}
	$u_1\in D_\Phi\cap dom(\phi(t)t)$.
\end{lemma}
\begin{proof}
	Making $v=0$ in \eqref{sol-u_1}, using $(f_0)$ and H\"{o}lder's inequality we find
	\begin{eqnarray}
		\int_{\Omega}\Phi(|\nabla u_1|)\,dx = Q(u_1)  \leq \lambda\int_{\Omega} f(x,u_1)u_1\,dx<\infty.	\nonumber
	\end{eqnarray}
	This proves that $u_1\in D_\Phi$.
	
	In the sequel, we will show that $u_1 \in \text{dom}(\phi(t)t)$.
	Setting $v=\left(1-\frac{1}{n}\right)u_1$ in \eqref{sol-u_1}, we get
	$$
	\begin{array}{l}
	\displaystyle \int_{\Omega}\left[\Phi(|\nabla (1-{1}/{n})u_1|)-\Phi(|\nabla u_1|)\right]\,dx=Q\left((1-{1}/{n})u_1\right)-Q(u_1)  \leq 	\displaystyle-\frac{1}{n}\lambda \int_{\Omega}f(x,u_1)u_1\,dx.
	\end{array}	
	$$
	Since $\Phi$ is $C^{1}$, there exists $\theta_n(x) \in [0,1]$ such that
	$$
	\frac{\Phi(|\nabla u_1-\frac{1}{n}\nabla u_1|)-\Phi(|\nabla u_1|)}{-\frac{1}{n}}=\phi(|(1-{\theta_n(x)}/{n})\nabla u_1|)(1-{\theta_n(x)}/{n})|\nabla u_1|^{2}.
	$$
	Recalling that $0< 1-{\theta_n(x)}/{n}\leq 1$, we know that
	$$
	1-{\theta_n(x)}/{n} \geq (1-{\theta_n(x)}/{n})^{2}:=g_n(x),
	$$
	which leads to
	$$
	\int_{\Omega}\phi(|g_n(x)\nabla u_1|)|g_n(x)\nabla u_1|^{2}\,dx \leq \lambda\int_{\Omega}f(x,u_1)u_1\,dx, \quad \forall n \in \mathbb{N}.
	$$
Letting  $n \to +\infty$, we derive that
	
	$$
	\int_{\Omega}\phi(|\nabla u_1|)|\nabla u_1|^{2}\,dx\leq\lambda\int_{\Omega}f(x,u_1)u_1\,dx.
	$$
	Recalling that
	$$
	\phi(t)t^{2}=\Phi(t)+\widetilde{\Phi}(\phi(t)t), \quad \forall t \in \mathbb{R}
	$$
	so, we have
	$$
	\phi(|\nabla u_1|)|\nabla u_1|^{2}=\Phi(|\nabla u_1|)+\widetilde{\Phi}(\phi(|\nabla u_1|)|\nabla u_1|),
	$$
	which leads to
	$$
	\int_{\Omega}\phi(|\nabla u_1|)|\nabla u_1|^{2}\,dx=\int_{\Omega}\Phi(|\nabla u_1|)\,dx+\int_{\Omega}\widetilde{\Phi}(\phi(|\nabla u_1|)|\nabla u_1|)\,dx.
	$$
As $\displaystyle \int_{\Omega}\phi(|\nabla u_1|)|\nabla u_1|^{2}\,dx$ and $\displaystyle \int_{\Omega}\Phi(|\nabla u_1|)\,dx$ are finite, we infer that $\displaystyle \int_{\R^N}\widetilde{\Phi}(\phi(|\nabla u_1|)|\nabla u_1|)\,dx$ is also finite, then $u_1 \in \text{dom}(\phi(t)t)$. This finishes the proof.	
\end{proof}

\begin{lemma} \label{Lema0}
	Suppose $(\phi_{1})-(\phi_{2})$ and let $u \in  D_\Phi $ be a critical point of $I$. If $u \in \text{dom}(\phi(t)t)$, then it is a weak solution for $(P)$, that is,
	$$
	\int_{\mathbb{R}^N}\phi(|\nabla u|)\nabla u \nabla v\,dx =\lambda\int_{\mathbb{R}^N}f(x,u)v\,dx,\quad \forall v \in W_0^{1,\Phi}(\Omega).
	$$	
\end{lemma}
\begin{proof}
	Hereafter, we adapt the arguments found in  \cite[Lemma 4.5]{EGS}. Given $\epsilon \in (0,\frac{1}{2})$ and $v \in C_0^{\infty}(\R^N)$, we set the function
	$$
	v_\epsilon =\frac{1}{1-\frac{\epsilon}{2}}((1-\epsilon)u+\epsilon v).
	$$	
	Hence, as $u$ is a critical point of $I$,
	$$
	\int_{\mathbb{R}^N}\Phi(|\nabla v_\epsilon|)\,dx - \int_{\R^N}\Phi(|\nabla u|)\,dx \geq \lambda\int_{\R^N}f(x,u)(v_\epsilon-u)\,dx, \forall \epsilon \in (0,{1}/{2}),
	$$	
	and so,
	$$
	\frac{\int_{\mathbb{R}^N}\Phi(|\nabla v_\epsilon|)\,dx- \int_{\R^N}\Phi(|\nabla u|)\,dx}{\epsilon}  \geq \lambda\int_{\R^N}f(x,u)\left(\frac{v_\epsilon-u}{\epsilon}\right)\,dx.
	$$	
	 Taking the limit as $\epsilon \to 0$, we get
	$$
	\int_{\R^N}\phi(|\nabla u|)\nabla u (\nabla v -\nabla u/2) \, dx  \geq \lambda\int_{\R^N}f(x,u)(v-u/2)\,dx
	$$
	or equivalently
	$$
	\int_{\R^N}\phi(|\nabla u|)\nabla u \nabla v  \, dx -\lambda\int_{\R^N}f(x,u)v\,dx \geq A, \quad \forall v \in C_0^{\infty}(\R^N),
	$$
	where
	$$
	A=\frac{1}{2}\int_{\R^N}\phi(|\nabla u|)|\nabla u|^2   \, dx -\frac{\lambda}{2}\int_{\R^N}f(x,u)u\,dx.
	$$
	As $C_0^{\infty}(\R^N)$ is a vector space, the last inequality gives
	$$
	\int_{\R^N}\phi(|\nabla u|)\nabla u \nabla v  \, dx -\lambda\int_{\R^N}f(x,u)v\,dx=0, \quad \forall v \in C_0^{\infty}(\R^N).
	$$
	
	Now  the result follows using the weak$^*$ density of $C_0^{\infty}(\R^N)$ in $W^{1,\Phi}(\R^N)$ together with the fact that $\phi(|\nabla u|)|\nabla u|  \in L^{\widetilde{\Phi}}(\R^N)$.	
\end{proof}

As a byproduct of the last lemma is the following corollary.

\begin{corollary}\label{sol1}
	Suppose that $(\phi_{1})-(\phi_{3})$, $(f_0)$ and $(f_2)$ hold. Then $u_1$ is a solution of problem $(P)$ with $I(u_1)<0$ for $\lambda \geq \lambda^*$.
\end{corollary}

We are going to use the Mountain Pass Theorem to find a second
critical point of $I$. To this end, we define
$$
g(x,t):=
\left\{
\begin{array}{ll}
	0 & \mbox{if } t<0;\\
	f(x,t) & \mbox{if }0\leq t\leq u_1(x);\\
	f(x,u_1) & \mbox{if } t>u_1(x).
\end{array}
\right.
$$
Now, let us consider the functional  $J:W_0^{1,\Phi}(\Omega) \to \mathbb{R}\cup\{+\infty\}$ defined by
\begin{eqnarray}\label{J}
	J(u)=\int_{\Omega}\Phi(|\nabla u|)dx-\lambda\int_{\Omega}G(x,u)dx,
\end{eqnarray}
where $G(x,t):=\int_0^tg(x,s)ds$. Due to Lemma \ref{Lema0}, we can follow the same ideas found in \cite[Lemma 2]{MRep} to prove the lemma below

	\begin{lemma}\label{Auxiliar}
	If $u$ is a solution of problem
	$$
	\left\{
	\begin{array}{ll}
	\displaystyle-\Delta_{\Phi}{u}=\lambda g(x,u), \quad \mbox{in} \quad \Omega, \\
	u =0,\qquad\mbox{on}\qquad\partial\Omega,
	\end{array}
	\right.
	\leqno{(PA)}
	$$
	then $u\leq u_1$.	
\end{lemma}

The next lemma establishes the first mountain pass geometry.
\begin{lemma}\label{MPG}
	Suppose that $(\phi_{1})-(\phi_{3})$, $(f_0)$ and $(f_1)$ hold. There exist $r,\rho>0$ such that $J(u)\geq \rho$ for all $u\in W_0^{1,\Phi}(\Omega)$ with $\|u\|=r$.
\end{lemma}
\begin{proof}
	Condition $(f_1)$ implies that there exists $\delta>0$ such that
	$$F(x,t)\leq 0,~0\leq t\leq \delta.$$
	Define $\Omega_u:=[u>\min\{\delta,u_1\}]$. From the last inequality,
	$$G(x,u(x))=F(x,u(x))\leq 0,~\mbox{a.e. in }x\in \Omega\setminus\Omega_u.$$
	On the other hand,
	$$G(x,u(x))\leq 0,~\mbox{a.e. in }\Omega_u\cap[u_1<u<\delta].$$
	 Let  $\Omega_{u,\delta}:=\Omega_u\setminus [u_1<u<\delta]$, $\Omega_{u,\delta}^-:=\Omega_{u,\delta}\cap [u\leq u_1]$ and  $\Omega_{u,\delta}^+:=\Omega_{u,\delta}\cap [u> u_1]$. From $(f_0)$,  $W^{1,\Phi}_0(\Omega)\hookrightarrow L^\Phi(\Omega)\hookrightarrow L^l(\Omega)\hookrightarrow L^s(\Omega)$ for  $s\in(m_A,l)$. Then,
	\begin{eqnarray}\label{des}
		\lambda\int_{\Omega_{u,\delta}}G(x,u)dx &\leq& \lambda C\int_{\Omega_{u,\delta}^-}(A(|u|)+|u|)\,dx+\lambda \int_{\Omega_{u,\delta}^+}(F(x,u_1)+f(x,u_1)(u-u_1))\,dx\nonumber\\
		&\leq&\lambda \overline{C}\int_{\Omega_{u,\delta}}(A(|u|)+|u|)\,dx \nonumber\\
		&\leq&\lambda C_\delta\int_{\Omega_{u,\delta}}\max\left\{\left(\frac{u}{\delta}\right),\left(\frac{u}{\delta}\right)^{m_A}\right\} \,dx\nonumber\\
		&\leq&\lambda \frac{C_\delta}{\delta^s} \int_{\Omega_{u,\delta}} |u|^s\,dx \leq \overline{C}_\delta\|u\|^s,
	\end{eqnarray}
	where $\overline{C},~C_\delta$ and $\overline{C}_\delta$ are positive constants. Therefore, considering $\|u\|\geq 1$, by Lemma \ref{ed1} and \eqref{des},
	$$J(u)\geq \|u\|^l-\lambda\int_{\Omega_{u,\delta}}G(x,u)dx\geq \|u\|^l(1-\lambda\overline{C}_\delta \|u\|^{s-l}).$$
	This proves the lemma.
\end{proof}

\begin{lemma}\label{jc}
	$J$ is coercive.
\end{lemma}
\begin{proof}
By $(f_0)$
$$|F(x,t)|\leq C_A(A(t)+|t|),~t\in\mathbb{R}$$
and
$$f(x,u_1)(t-u_1)\leq 2C_A(A(t)+|t|),~\forall t\geq u_1. $$
Thus, for $\|u\|\geq 1$,

$\displaystyle\int_\Omega\Phi(|\nabla u|)dx-\lambda\int_\Omega G(x,u)dx=$
\begin{eqnarray}
&=  & \int_\Omega\Phi(|\nabla u|)dx-\lambda\int_{u\leq u_1}F(x,u)dx-\lambda\int_{u>u_1}F(x,u_1)+f(x,u_1)(u-u_1)dx\nonumber \\
&\ge& \int_\Omega\Phi(|\nabla u|)dx-3\lambda C_A\int_{\Omega}(A(u)+|u|)dx\nonumber\\
%
&\geq  & \|u\|^l-\lambda c_1\|u\|^{m_A}-\lambda c_2\|u\|,
\end{eqnarray}
 where $c_1,c_2$ are positive constants. As $1<m_A<l$, we get the desired result.
\end{proof}

\noindent{\bf Proof of Theorem \ref*{T1}: } Gathering Lemma \ref{MPG} with the fact that $J(u_1)=I(u_1)<0$ for $\lambda \geq \lambda^*$, we can apply the Mountain Pass Theorem found in \cite[Theorem 3.1]{AlvesdeMorais} to guarantee the existence of a $(PS)$ sequence $(u_n)\subset W_0^{1,\Phi}(\Omega)$ associated with the mountain pass level of $J$, that is, $J(u_n)\to c\geq \rho$ and $\tau_n\to 0$ in $\R$ such that
\begin{equation} \label{sequencia2}
Q(v)-Q(u_n) \geq \lambda\int_{\Omega}g(x,u_n)(v-u_n)\,dx- \tau_n\|v-u_n\|,
\end{equation}
for all $v \in W_0^{1,\Phi}(\Omega)$ holds true for  all $n \in\mathbb{N}$,
where
$$c:=\inf_{\gamma\in\Gamma}\max_{t\in[0,1]}J(\gamma(t))$$
and
$$\Gamma=\{\gamma\in C[0,1]|~\gamma(0)=0,~\gamma(1)=u_1\}.$$
By Lemma \ref{jc}, $(u_n)$ is bounded. Hence, we can assume without loss of generality that $u_n\stackrel{*}\rightharpoonup u_2$ in $W_0^{1,\Phi}(\Omega)$.

In the sequel, we will show that $u_2 \in \text{dom}(\phi(t)t)$. By Lemma \ref{dominio}, there is $(v_n) \subset \text{dom}(\phi(t)t)$ such that
$$
|v_n|\leq |u_n|,\qquad\|v_n-u_n\| \leq 1/n \quad \mbox{and} \quad \int_{\Omega}\Phi(|\nabla v_n|)\,dx \leq \int_{\Omega}\Phi(|\nabla u_n|)\,dx, \quad \forall n \in \mathbb{N}.
$$
Consequently,
$$
Q(v)-Q(v_n) \geq \lambda\int_{\Omega}g(x,u_n)(v-u_n)\,dx- |\tau_n|\, \|v-u_n\|,
$$
for all $v \in W_0^{1,\Phi}(\Omega)$. Setting $v=v_n-\frac{1}{n}v_n$, we get
$$
Q(v_n-\frac{1}{n}v_n)-Q(v_n)
\geq \lambda\int_{\Omega}g(x,u_n)(v_n-\frac{1}{n}v_n-u_n)\,dx- |\tau_n| \, \|v_n-\frac{1}{n}v_n-u_n\|,
$$
that is,
$$
\begin{array}{l}
\displaystyle		\int_{\Omega}\frac{(\Phi(|\nabla v_n-\frac{1}{n}\nabla v_n|)-\Phi(|\nabla v_n|))}{-\frac{1}{n}}\,dx \leq  \\
\mbox{}\\
\displaystyle  -n\lambda\int_{\Omega}g(x,u_n)(v_n-u_n)\,dx+ \lambda\int_{\Omega}g(x,u_n)v_n \,dx \nonumber + n|\tau_n|\,\|v_n-u_n\|+|\tau_n|\,\|v_n\|. \nonumber
\end{array}
$$	
As $(u_n)$ is bounded in $W_0^{1,\Phi}(\Omega)$, $(g(x,u_n))$ is bounded in $L_{\widetilde{A}}(\Omega)$, $(\tau_n)$ is bounded in $\mathbb{R}$ and $\|v_n-u_n\| \leq \frac{1}{n}$, it follows that the right side of the above inequality is bounded. Therefore, there is $M>0$ such that
$$
\int_{\Omega}\frac{(\Phi(|\nabla v_n-\frac{1}{n}\nabla v_n|)-\Phi(|\nabla v_n|))}{-\frac{1}{n}}\, dx  \leq M, \quad \forall n \in \mathbb{N}.
$$
Using again that $\Phi$ is $C^{1}$, there exists $\theta_n(x) \in [0,1]$ such that
$$
\frac{\Phi(|\nabla v_n-\frac{1}{n}\nabla v_n|)-\Phi(|\nabla v_n|)}{-\frac{1}{n}}=\phi(|(1-{\theta_n(x)}/{n})\nabla v_n|)(1-{\theta_n(x)}/{n})|\nabla v_n|^{2}.
$$
Recalling that $0< 1-{\theta_n(x)}/{n}\leq 1$, we know that
$$
1-{\theta_n(x)}/{n} \geq (1-{\theta_n(x)}/{n})^{2},
$$
which leads to
$$
\int_{\Omega}\phi(|(1-{\theta_n(x)}/{n})\nabla v_n|)(1-{\theta_n(x)}/{n})^{2}|\nabla v_n|^{2}\,dx \leq M \quad \forall n \in \mathbb{N}.
$$
As $u_n \stackrel{*}{\rightharpoonup} u_2$ in $W_0^{1,\Phi}(\Omega)$, we also have $(1-{\theta_n(x)}/{n})v_n \stackrel{*}{\rightharpoonup} u_2$ in $W_0^{1,\Phi}(\Omega).$ Therefore,  using the fact that $\phi(t)t^{2}$ is convex, we can apply \cite[Lemma 3.2]{GKMS} to obtain
$$
\liminf_{n \to +\infty}\int_{\Omega}\phi(|(1-{\theta_n(x)}/{n})\nabla v_n|)(1-{\theta_n(x)}/{n})^{2}|\nabla v_n|^{2}\,dx  \geq \int_{\Omega}\phi(|\nabla u_2|)|\nabla u_2|^{2}\,dx,
$$
and so,
$$
\int_{\Omega}\phi(|\nabla u_2|)|\nabla u_2|^{2}\,dx \leq M.
$$
Recalling that
$$
\phi(t)t^{2}=\Phi(t)+\widetilde{\Phi}(\phi(t)t), \quad \forall t \in \mathbb{R},
$$
we have
$$
\phi(|\nabla u_2|)|\nabla u_2|^{2}=\Phi(|\nabla u_2|)+\widetilde{\Phi}(\phi(|\nabla u_2|)|\nabla u_2|),
$$
which leads to
$$
\int_{\Omega}\phi(|\nabla u_2|)|\nabla u_2|^{2}\,dx=\int_{\Omega}\Phi(|\nabla u_2|)\,dx+\int_{\Omega}\widetilde{\Phi}(\phi(|\nabla u_2|)|\nabla u_2|)\,dx.
$$
Since $\displaystyle \int_{\Omega}\phi(|\nabla u_2|)|\nabla u_2|^{2}\,dx$ and $\displaystyle \int_{\Omega}\Phi(|\nabla u_2|)\,dx$ are finite, we see that $\displaystyle \int_{\Omega}\widetilde{\Phi}(\phi(|\nabla u_2|)|\nabla u_2|)\,dx$ is also finite. \\
{\bf Claim:} $J(u_n)\to J(u_2)$ as $n \to +\infty$. \\

In fact, from \eqref{sequencia2} and $W_0^{1,\Phi}(\Omega)\stackrel{comp}\hookrightarrow L_A(\Omega)$,
$$\int_{\Omega}G(x,u_n) dx\to \int_{\Omega}G(x,u_2) dx\qquad\mbox{and}\qquad \int_{\Omega}g(x,u_n)(v-u_n) dx\to \int_{\Omega}g(x,u_2)(v-u_2) dx.$$
Since $(J(u_n))$ is bounded sequence, for some a subsequence of $(u_n)$, still denoted by itself, we can assume that
$$\lim_{n\to \infty}Q(u_n)= L.$$
Using the fact  that $Q$ is lower semicontinuous with respect to the weak* topology, we derive that
$$Q(u_2)\leq \liminf_{n\to\infty} Q(u_n)=L.$$
On the other hand, making $v=u_2$ in \eqref{sequencia2}, we conclude that
$$
Q(u_2)\geq  \liminf_{n\to\infty} Q(u_n)=L,
$$
and so,
$$\lim_{n\to \infty} Q(u_n)=Q(u_2).$$
Thus, $J(u_n)\to J(u_2)=c$. Now, we can use the same ideas of Lemma \ref{Domin} to prove that $u_2\in D_\Phi\cap dom(\phi(t)t)$. Letting $n \to +\infty$  in \eqref{sequencia2}, we conclude that $u_2$ is a critical point of $J$. Using similar idea explored in the proof of Lemma \ref{Lema0}, $u_2$ is a weak solution of $(PA)$.  From Lemma \ref{Auxiliar} we obtain that $u_2\leq u_1$, then $g(x,u_2(x))=f(x,u_2(x))$ for all $x \in \Omega$, from where it follows that $u_2$ is a weak solution of
$(P)$ with $J(u_2)=I(u_2)$. Moreover, by Lemma \ref{Lema0}, $u_2$ is a weak solution of
$(P)$. But, from Lemma \ref{negative}, $I(u_1)<0<c=I(u_2)$, from where it follows that $u_1\neq u_2$. This finishes the proof.

\section{Proof of Theorem \ref{T2}}
In this section we will use the same approach of the last section. In order to avoid some repetitions, we are going to show only the different accounts. For example, it is important to point out that if $(\phi_4)$ holds, then $D_\Phi=dom(t\phi(t))=W_0^{1,\Phi}(\Omega)$.
\begin{lemma}\label{I-coercive-1}
	Assume that $(\phi_1),~(\phi_2),~(\phi_4)$ and $(f_3)$ hold. Then the functional $I$ is coercive.
\end{lemma}
\begin{proof}
	Initially, from  $(f_3)$,
	\begin{eqnarray}\label{cons-f_3}
		|F(x,t)|\leq \frac{C}{\alpha}\Phi(t)^\alpha.
	\end{eqnarray}
Using \eqref{cons-f_3}, H\"{o}lder's and Poincar\'e's Inequalities, there exist positive constants $C_1$ and $C_2$ satisfying
	\begin{eqnarray}
	I(u)&\geq&\int_{\Omega}\Phi(|\nabla u|)dx - \lambda \frac{C}{\alpha}\int_{\Omega}[\Phi(u)]^\alpha dx\nonumber\\
        &\geq& \int_{\Omega}\Phi(|\nabla u|)dx - \lambda C_1 \left(\int_{\Omega}\Phi(|\nabla u|) dx\right)^{\alpha}\nonumber\\
		&\geq&\int_{\Omega}\Phi(|\nabla u|)dx\left[1 - \lambda C_2\left(\int_{\Omega}\Phi(|\nabla u|) dx\right)^{\alpha-1}\right].\nonumber
	\end{eqnarray}
As $\alpha \in (0,1)$, the Lemma \ref{ed1} ensures that $I$ is coercive functional. This finishes the proof.
\end{proof}

\begin{lemma}\label{negative-1}
	Assume that $(\phi_1),~(\phi_2),~(\phi_4),~(f_2)$ and $(f_3)$ hold. Then, there exist $\lambda_*>0$ such that $I$ is bounded from below in  $W^{1,\Phi}_0(\Omega)$ and  $\displaystyle \inf_{u \in W_0^{1,\Phi}(\Omega)} I(u)<0$ for all $\lambda>\lambda_*$.
\end{lemma}
\begin{proof}
Since the boundedness of $I$ from below follows as in Lemma \ref{negative}, we will omit this part. The rest of the proof follows as the same ideas of Lemma \ref{negative}. But, we need to change the inequality \eqref{Negative} by
	\begin{eqnarray}
		\int_{\Omega}F(x,u_0)dx
		&\geq & \int_{\Omega_1}F(x,t_1)dx-C\int_{\Omega\setminus\Omega_1}\Phi(u_0)^\alpha dx\nonumber\\
		&\geq & \int_{\Omega_1}F(x,t_1)dx-C|\Omega\setminus\Omega_1|\Phi(t_1)^\alpha >0.
	\end{eqnarray}
	Here, we point out that in the first inequality we have used \eqref{cons-f_3}.
\end{proof}

\begin{corollary}\label{sol12}
	Suppose that $(\phi_{1}),~(\phi_{2}),~(\phi_{4})$, $(f_2)$ and $(f_3)$ hold. Then, there exists a solution $u_1$ of problem $(P)$ such that $I(u_1)<0$.
\end{corollary}
\begin{proof}
	Since $I$ is lower semicontinuous in weak$^*$ topology, we can use Lemmas \ref{I-coercive-1} and \ref{negative-1} to obtain $u_1\in W_0^{1,\Phi}(\Omega)$ such  that
	$$I(u_1):=\inf_{u\in W_0^{1,\Phi}(\Omega)}I(u)<0.$$
	Since $I$ is G\^ateaux differentiable, it follows from \cite[Prop. 1.1]{Szulkin} that $u_1$ is a critical point of I. Moreover, using the fact that $\Phi$ satisfies the $\Delta_2$-condition, we obtain that $D_\Phi=dom(\phi(t)t)=W_0^{1,\Phi}(\Omega)$. From Lemma \ref{Lema0}, we conclude that $u_1$ is a solution of problem $(P)$.
\end{proof}

Now, we shall consider the functional $J$ defined by \eqref{J}.

\begin{lemma}\label{MPG2}
	Assume $(\phi_1),~(\phi_2),~(\phi_4),~(f_2)$, $(f_3)$ and $m<l^*$. Then there exist $r,\rho>0$ such that $I(u)\geq \rho$ for all $u\in W_0^{1,\Phi}(\Omega)$ with $\|u\|=r$.
\end{lemma}

\begin{proof}
	We shall use the same ideas of Lemma \ref{MPG}, but we need to do some adjusts. Assume that $\|u\|\leq 1$. We will change the inequality \eqref{des}. For this end, by $(f_3)$, \eqref{cons-f_3} and choosing $s\in(m, l^*)$,
	\begin{eqnarray}\label{des1}
	\lambda\int_{\Omega_{u,\delta}}G(x,u)dx &\leq& \lambda C\int_{\Omega_{u,\delta}^-}\Phi(u)^\alpha\,dx+\lambda C\int_{\Omega_{u,\delta}^+}(F(x,u_1)+f(x,u_1)(u-u_1))\,dx\nonumber\\
	&\leq&\lambda \overline{C}\int_{\Omega_{u,\delta}}\Phi(u)^\alpha\,dx \nonumber\\
	&\leq&\lambda C_\delta\int_{\Omega_{u,\delta}}\max\left\{\left(\frac{u}{\delta}\right)^\alpha,\left(\frac{u}{\delta}\right)^{m\alpha}\right\} \,dx\nonumber\\
	&\leq&\lambda \frac{C_\delta}{\delta^s} \int_{\Omega_{u,\delta}} |u|^s\,dx \leq \overline{C}_\delta\|u\|^s.
	\end{eqnarray}
	Thus, from \eqref{des1},
	$$J(u)\geq \|u\|^m-\lambda\int_{\Omega_{u,\delta}}G(x,u)dx\geq \|u\|^m(1-\lambda\overline{C}_\delta \|u\|^{s-m}).$$
	
\end{proof}

The Lemma \ref{MPG2} combined with the equality $J(u_1) =I(u_1)<0$ permit to use again the Mountain Pass Theorem to obtain
a sequence $(u_n)\subset W_0^{1,\Phi}(\Omega)$ such that $J(u_n)\to c\geq \rho>0$ and \eqref{sequencia2} holds.

\begin{lemma}\label{Bounded}
	Assume $(\phi_1),(\phi_2),(\phi_4)$ and $(f_{1})- (f_{3})$. Then $(u_n)$ is bounded.
\end{lemma}
\begin{proof}
	Using similar idea of Lemmas \ref{I-coercive-1} and \ref{jc}, we can prove that $J$ is coercive, and so, $(u_n)$ must be bounded.
\end{proof}

	\noindent {\bf Proof of Theorem \ref{T2}: } Initially, by Lemma \ref{Bounded}, $u_n\stackrel{*}\rightharpoonup u_2$, for some $u_2\in W_0^{1,\Phi}(\Omega)$. Moreover, $u_n\to u_2$ a.e. in $\Omega$ and there exists $h\in L^\Phi(\Omega)$ such that $|u_n|\leq h$.\\	
	{\bf Claim:} $J(u_n)\to J(u_2)$ as $n \to +\infty$.
	
	Firstly, we will prove that
	\begin{eqnarray}\label{conv}
		\int_{\Omega}G(x,u_n) dx\to \int_{\Omega}G(x,u_2) dx.
	\end{eqnarray} 	
	In fact, by \eqref{cons-f_3},	
	$$|G(x,t)|=|F(x,t)|\leq C\Phi^\alpha(t),~0<t<u_1.$$
	Now, using \cite[Theorem 1.6]{SCA}, we have $u_1\in L^{\infty}(\Omega)$. From this,
$$|G(x,t)| \leq |F(x,u_1)|+ |f(x,u_1)||t-u_1|\leq C_1+C_2t, \quad u_1<t.$$
	Consequently,
$$|G(x,u_n)|\leq C_1 + C\Phi^\alpha(h)+C_2h \in L^{1}(\Omega).$$
Now, (\ref{conv}) follows from Lebesgue dominated convergence theorem. A similar argument works to prove that
	$$\int_{\Omega}g(x,u_n)(v-u_n) dx\to \int_{\Omega}g(x,u_2)(v-u_2) dx,~v\in W_0^{1,\Phi}(\Omega)$$
As in the proof of Theorem \ref{T1},
	$$\lim_{n\to \infty} Q(u_n)=Q(u_2).$$
	Thus, $J(u_n)\to J(u_2)=c$. Letting $n \to +\infty$ in \eqref{sequencia2}, we conclude that $u_2$ is a critical point of $J$. Using the same ideas explored in the proof of Lemma \ref{Lema0}, it is easy to see that $u_2$ is a weak solution of $(PA)$.  Arguing as in the proof of Theorem \ref{T1}, $u_2$ is a weak solution of
	$(P)$ with $J(u_2)=I(u_2)$. But, from Lemma \ref{negative}, $I(u_1)<0<c=I(u_2)$, this implies that $u_1\neq u_2$. Moreover, From \cite[Theorem 1.6]{SCA},  $u_2\in L^{\infty}(\Omega)$, which finishes the proof.

\end{document}